\documentclass[10pt,oneside,a4paper]{article} 				

\usepackage[fleqn]{amsmath} 										
\usepackage{amsthm,amsfonts,latexsym,amssymb,amscd} 				
\usepackage[mathscr]{eucal}   										
\usepackage[all,2cell]{xy} \UseAllTwocells 							
\usepackage{stmaryrd} 												
\usepackage{framed}
\usepackage{color}
\definecolor{cmyk}{cmyk}{.0,.4,.9,.5}								
\usepackage{txfonts}
\usepackage{mathtools} 												

\usepackage{bbold}

\usepackage{scalerel} 	

\usepackage{tikz}
\usetikzlibrary{svg.path}

\definecolor{orcidlogocol}{HTML}{A6CE39}
\tikzset{
  orcidlogo/.pic={
    \fill[orcidlogocol] svg{M256,128c0,70.7-57.3,128-128,128C57.3,256,0,198.7,0,128C0,57.3,57.3,0,128,0C198.7,0,256,57.3,256,128z};
    \fill[white] svg{M86.3,186.2H70.9V79.1h15.4v48.4V186.2z}
                 svg{M108.9,79.1h41.6c39.6,0,57,28.3,57,53.6c0,27.5-21.5,53.6-56.8,53.6h-41.8V79.1z M124.3,172.4h24.5c34.9,0,42.9-26.5,42.9-39.7c0-21.5-13.7-39.7-43.7-39.7h-23.7V172.4z}
                 svg{M88.7,56.8c0,5.5-4.5,10.1-10.1,10.1c-5.6,0-10.1-4.6-10.1-10.1c0-5.6,4.5-10.1,10.1-10.1C84.2,46.7,88.7,51.3,88.7,56.8z};
  }
}

\newcommand\orcidicon[1]{\href{https://orcid.org/#1}{\mbox{\scalerel*{
\begin{tikzpicture}[yscale=-1,transform shape]
\pic{orcidlogo};
\end{tikzpicture}
}{|}}}}

\usepackage[draft=false,breaklinks,backref]{hyperref} 				
\hypersetup{
           breaklinks=true,   
           pdfusetitle=true,  
        }

\usepackage{cite}

\newcommand{\Ug}{\mathfrak{U}}

\newcommand{\NN}{{\mathbb{N}}}

\newcommand{\As}{{\mathscr{A}}}\newcommand{\Bs}{{\mathscr{B}}}\newcommand{\Cs}{{\mathscr{C}}}
\newcommand{\Ds}{{\mathscr{D}}}
\newcommand{\Es}{{\mathscr{E}}}

\newcommand{\Ks}{{\mathscr{K}}} 
\newcommand{\Ms}{{\mathscr{M}}}\newcommand{\Ns}{{\mathscr{N}}}\newcommand{\Os}{{\mathscr{O}}}
\newcommand{\Qs}{{\mathscr{Q}}}\newcommand{\Rs}{{\mathscr{R}}}
\newcommand{\Us}{{\mathscr{U}}}
\newcommand{\Xs}{{\mathscr{X}}}

\DeclareFontFamily{U}{rsfs}{\skewchar\font127 }
\DeclareFontShape{U}{rsfs}{m}{n}{%
   <5> <6> rsfs5
   <7> rsfs7
   <8> <9> <10> <10.95> <12> <14.4> <17.28> <20.74> <24.88> rsfs10
}{}
\DeclareSymbolFont{rsfs}{U}{rsfs}{m}{n} 
\DeclareSymbolFontAlphabet{\scr}{rsfs}

\newcommand{\Kf}{\scr{K}} 
\newcommand{\Qf}{\scr{Q}}
\newcommand{\Xf}{\scr{X}}

\renewcommand{\emph}{\textbf} 										
\newcommand{\imp}{\Rightarrow}										
\newcommand{\st}{\ : \ }											

\newcommand{\dya}{\scaleobj{0.8}{\Join}}
\newcommand{\ddya}{\scaleobj{0.8}{\mathrlap{\hspace{0.4pt}\Join}\bigcirc}}
\newcommand{\Xddya}{\overleftrightarrow{\kern -2pt\Xf}_{\kern -2pt\As}^{\ddya}}
\newcommand{\Xdya}{\overleftrightarrow{\kern -2pt\Xf}_{\kern -2pt\As}^{\dya}}
\newcommand{\dpartial}{\partial\kern-4.5pt /}
\newcommand{\yad}{\ \scaleobj{0.6}{\pmb{|}} \kern-2.5pt \scaleobj{0.8}{\Join}}

\newcommand{\hlink}[2]{\href{#1}{\texttt{#2}}} 						

\newcommand{\xqedhere}[2]{%
  \rlap{\hbox to#1{\hfil\llap{\ensuremath{#2}}}}}

\theoremstyle{plain}
\newtheorem{theorem}{Theorem}[section]							
  
\newtheorem{corollary}[theorem]{Corollary}

\newtheorem{lemma}[theorem]{Lemma}

\newtheorem{definition}[theorem]{Definition}

\theoremstyle{definition}

\numberwithin{equation}{section}  									
\usepackage{blindtext}

\title{\textbf{Involutive Weak Cubical $\omega$-categories}
}

\author{\normalsize  
\orcidicon{0000-0002-3891-7717} Paratat Bejrakarbum$^a$ \quad
\orcidicon{0000-0002-1387-9283} Paolo Bertozzini$^{b} \footnote{Currently unaffiliated independent reseacher based in Bangkok.}$ \quad  
\orcidicon{} Supaporn Theesoongnern$^c$ \footnote{
Corresponding ``first'' author. 
Notice that, contrary to the published paper, the authors appear here in the standard alphabetical order. 
}   
\\  
\normalsize \textit{Department of Mathematics and Statistics, Faculty of Science and Technology,}
\\
\normalsize \textit{Thammasat University, Pathumthani 12121, Thailand}
\\
\normalsize e-mail: 
\ 
$^a$ \texttt{paratat@tu.ac.th}
\quad 
$^b$ \texttt{paolo.th@gmail.com} 
\quad 
$^c$ \texttt{boomsup.t@gmail.com}
}

\date{\normalsize{
revised version: 
02 April 2025 \footnote{
This is a reformatted version, only for arXiv purposes, of a paper accepted for publication in \textit{Science and Technology Asia} 30(2). 
}  
\quad 
submitted version: 12 March 2024
\quad 
started: 17 February 2023 
}}

\begin{document}

\maketitle


\begin{abstract} \noindent 
We investigate the notion of involutive weak cubical $\omega$-categories via Penon's approach: as algebras for the monad induced by the free involutive strict $\omega$-category functor on cubical $\omega$-sets. 
A few examples of involutive weak cubical $\omega$-categories are provided. 

\medskip


\smallskip

\noindent
\emph{MCS-2020:} 
					18N65, 
					18N70, 
					18M40, 
					18N30, 
					18N99. 
\end{abstract}

\tableofcontents

\section{Introduction and Motivation}

Motivated by research in algebraic topology, category theory, starting from~\cite{EM45}, developed into an independent mathematical subject. Although higher categories had been already implicit in the definition of natural trasformations, the study of $n$-categories (both in their globular and cubical versions) was initiated in~\cite{Eh65}. Strict $\omega$-categories had been conjectured by J.Roberts (as later reported in \cite{Ro79}) and independently introduced and studied by~\cite{BH77}. 

\medskip 

The development of weak higher category theory (somehow implicit in the definition of monoidal category) probably started with the definition of bicategory in~\cite{Be67} and $n$-category in~\cite{St72} and is now a quite active area of research (see for example~\cite{ChLa04,Lei01,Lei04}). 

\medskip 

Algebraic approaches to the definition of weak globular higher-categories have been developed by~\cite{Ba98}, \cite{Pe99} and \cite{Lei04}. A similar study for the weak cubical higher categories, using Penon's technique, has been carried on by C.Kachour in several important recent works~\cite{Ka22}

\medskip 

The notion of involution (duality) in category theory has a relatively ``involved'' history with concepts independently introduced by several authors in different contexts and generality (see~\cite{BS09} and \cite[section~4]{BCLS20} for some bibliographical details); a recent systematic treatment of the topic is contained in~\cite{Ya20} where further references can be found. 

\medskip 

Here we are specifically interested in a (vertical) categorification of the usual $*$-operation in operator algebras: the ``$*$-categories'' considered in~\cite{GLR85}, \cite{Mi02} and the ``dagger categories''  axiomatized in~\cite{Se05} and utilized in~\cite{AC04}.  

\medskip 

Strict involutive globular $n$-categories have been considered in~\cite{BCLS20}. 
Weak involutive globular $\omega$-categories have been introduced, using Penon's contractions in~\cite{Be16,BeBe17} and, in~\cite{Be23,BeBe23}, using Leinster's definition of globular $\omega$-categories. 

\medskip 

In the present work, we aim at a sufficiently general definition of \textit{involutive weak cubical $\omega$-category} following the C.Kachour algebraic notion of weak cubical Penon $\omega$-category. 

\bigskip 

The organization of the paper is the following. 

\medskip 

After this introduction, in section~\ref{sec: strict}, we approach the study of strict involutive cubical $\omega$-categories: 
\begin{itemize}
\item 
following the ideas of~\cite{BH77} and~\cite{Ka22},   
suitably general notions of cubical $\omega$-quivers and cubical $\omega$-sets are introduced in definitions~\ref{def: quiver} and~\ref{def: set},  
\item 
self-dualities on cubical $\omega$-sets and the algebraic properties of cubical involutions are axiomatized, following the double category case in~\cite{BCDM14}, in definitions~\ref{def: operations} and~\ref{def: axinv}, 
\end{itemize}
  
The proof that the free strict involutive cubical $\omega$-category of a cubical $\omega$-set exists is postponed to section~\ref{sec: weak} in lemmata~\ref{lem: free-magma} and~\ref{lem: free-cat} and hence the associated monad is constructed in corollary~\ref{cor: monad}.  

\medskip 

In section~\ref{sec: weak} we deal with the involutive version of Penon-Kachour weak cubical $\omega$-categories: 
\begin{itemize}
\item 
we introduce in definition~\ref{def: contraction} a notion of Penon-Kachour contraction for our cubical $\omega$-sets,   
\item 
in lemma~\ref{lem: contraction} it is proved that the free contracted Penon-Kachour cubical involutive $\omega$-contraction exists and hence in theorem~\ref{th: weak} we show that we have an associated monad,  
\item 
in definition~\ref{def: weak} weak involutive cubical $\omega$-categories are introduced (similarly to Kachour for cubical groupoids) as algebras for the previous monad, 
\item 
some examples of such weak involutive cubical $\omega$-categories are suggested in subsection~\ref{ex: weak}. 
\end{itemize}
 
Finally in a brief outlook section~\ref{sec: outlook} we examine some possible future direction of development of this work. 

\section{Strict (Involutive) Cubical $\omega$-categories}\label{sec: strict}

The first definition only formalizes the idea that ``$n$-dimensional cells'' $x\in\Qs^n$ are equipped with a family of ``source/target'' $(n-1)$-dimensional cells, indexed as the ``faces of an $n$-dimensional hypercube''. The sets $D$ with cardinality $|D|=n$ indicate the possible ``directions'' of the $n$-dimensional cells, where the ``directions'' are selected via subsets (of cardinality $n$) in the infinite countable set $\NN_0$. In this generality, morphisms are just a countable family of ``dimension-preserving'' maps compatible with sources and targets. 
\begin{definition}\label{def: quiver}
An \emph{cubical $\omega$-quiver} is a family $\left(\Qs^n_{D-\{d\}}\xleftarrow{s^n_{D,d},\ t^n_{D,d}}\Qs^{n+1}_D\right)_{n\in\NN}$ of \emph{source maps} $s^n_{D,d}$ and \emph{target maps} $t^n_{D,d}$
indexed by $n\in\NN$, by any $D\subset \NN_0$ with cardinality $|D|=n+1$ and any $d\in D$. 

\medskip 

A \emph{morphism of cubical $\omega$-quivers} is a family $\Qs^n_{D}\xrightarrow{\phi^n_{D}}\hat{\Qs}^n_{D}$ indexed by $n\in\NN$ and $D\subset\NN$ with $|D|=n$, such that $\hat{s}^{n}_{D,d}\circ\phi^{n+1}_{D}=\phi^{n}_{D-\{d\}}\circ s^{n}_{D,d}$ and $\hat{t}^{n}_{D,d}\circ\phi^{n+1}_{D}=\phi^{n}_{D-\{d\}}\circ t^{n}_{D,d}$, for all $n\in\NN$, $D\subset\NN_0$ with $|D|=n$ and $d\in D$. 
\end{definition}

The actual $n$-dimensional ``cubical shape'' of $n$-cells is specified by the following axioms. 
\begin{definition}\label{def: set}
A \emph{cubical $\omega$-set} is a cubical $\omega$-quiver $\left(\Qs^n_{D-\{d\}}\xleftarrow{s^n_{D,d},\ t^n_{D,d}}\Qs^{n+1}_D\right)_{n\in\NN}$ satisfying the \emph{cubical axioms}:
\begin{gather*}
\forall n\in \NN \st n\geq 2, \quad \forall D\subset \NN_0\st |D|=n, \quad \forall d\neq e\in D, 
\\  
s^{n-2}_{D-\{d,e\}}\circ s^{n-1}_{D-\{d\}} = s^{n-2}_{D-\{d,e\}}\circ s^{n-1}_{D-\{e\}}, 
\quad \quad 
t^{n-2}_{D-\{d,e\}}\circ t^{n-1}_{D-\{d\}} = t^{n-2}_{D-\{d,e\}}\circ t^{n-1}_{D-\{e\}}, 
\\
s^{n-2}_{D-\{d,e\}}\circ t^{n-1}_{D-\{d\}} = t^{n-2}_{D-\{d,e\}}\circ s^{n-1}_{D-\{e\}}, 
\quad \quad 
t^{n-2}_{D-\{d,e\}}\circ s^{n-1}_{D-\{d\}} = s^{n-2}_{D-\{d,e\}}\circ t^{n-1}_{D-\{e\}}. 
\end{gather*}

A \emph{morphism of cubical $\omega$-sets} is just a morphism of underlying cubical $\omega$-quivers.
\end{definition}
A pictorial description of cubical $n$-cells, for four cases $n=0,D=\varnothing$; $n=1, D=\{1\}$; $n=2, D=\{1,2\}$; $n=3, D=\{1,2,3\}$ respectively, is here below: 
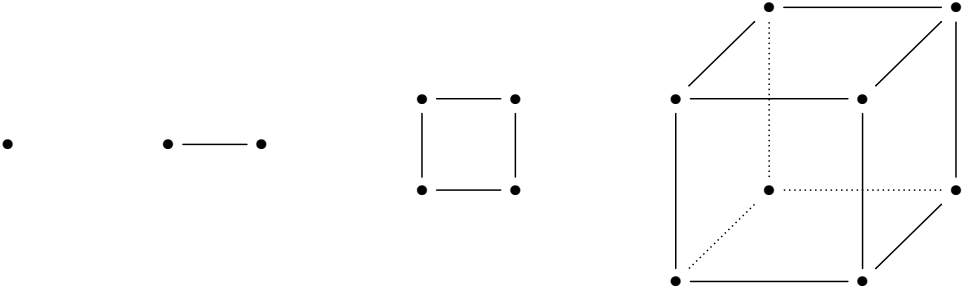
\begin{align*}
&\vcenter{\xymatrix{
\bullet 
}} 
& 
&\vcenter{\xymatrix{
\bullet \ar@{-}[r] & \bullet
}} 
& 
&\vcenter{\xymatrix{
\bullet \ar@{-}[r] \ar@{-}[d] & \bullet \ar@{-}[d]
\\
\bullet \ar@{-}[r] & \bullet
}} 
& 
&\vcenter{
\xymatrix{
& \bullet \ar@{-}[rr] \ar@{-}[dl] & & \bullet \ar@{-}[dl] \ar@{-}[dd]
\\
\bullet \ar@{-}[rr] \ar@{-}[dd] & & \bullet \ar@{-}[dd] & 
\\
& \bullet \ar@{.}[rr] \ar@{.}[dl] \ar@{.}[uu] & & \bullet \ar@{-}[dl]
\\
\bullet \ar@{-}[rr] & & \bullet & 
}}
\end{align*}

Next we introduce three families of (binary, nullary, unary) operations on cubical $n$-cells. 
\begin{definition} \label{def: operations}
Given a cubical $\omega$-set $\Qs$, we can introduce on it the following operations: 
\begin{itemize}
\item 
\emph{binary compositions} 
\begin{equation*}
\circ^n_{D,d}:\Qs^n_D \times_{\Qs^{n-1}_{D-\{d\}}} \Qs^n_D\to \Qs^n_D, 
\quad \forall n\in \NN_0 \quad \forall d\in D\subset \NN_0 \st |D|=n,
\end{equation*}
where $\Qs^n_D \times_{\Qs^{n-1}_{D-\{d\}}} \Qs^n_D:=\Big\{(x,y) \ | \ s^{n-1}_{D,d}(x)=t^{n-1}_{D,d}(y) \Big\}$ and we assume: 
\begin{gather*}
s^{n-1}_{D,d}(x\circ^n_{D,d} y)= s^{n-1}_{D,d}(y), \quad \quad 
t^{n-1}_{D,d}(x\circ^n_{D,d} y)= t^{n-1}_{D,d}(x),
\\ 
s^{n-1}_{D,e}(x\circ^n_{D,d} y)=s^{n-1}_{D,e}(x)\circ^{n-1}_{D-\{e\},d}s^{n-1}_{D,e}(y), 
\quad \quad 
t^{n-1}_{D,e}(x\circ^n_{D,d} y)=t^{n-1}_{D,e}(x)\circ^{n-1}_{D-\{e\},d}t^{n-1}_{D,e}(y), 
\quad \quad \forall e\neq d. 
\end{gather*}
\begin{equation*}
\vcenter{\xymatrix{
& \ar[rr]^{s^{n-1}_{D,e}(y)} \ar[d]_{s^{n-1}_{D,d}(y)} \ar@{:>}[drr]|{y} & & \ar[rr]^{s^{n-1}_{D,e}(x)} \ar[d]^{}_{} \ar@{:>}[drr]|{x} & & \ar[d]^{t^{n-1}_{D,d}(x)} 
\\
& \ar[rr]_{t^{n-1}_{D,e}(y)} & & \ar[rr]_{t^{n-1}_{D,e}(x)} & &
}}
\mapsto 
\vcenter{\xymatrix{
\ar[d]_{s^{n-1}_{D,d}(y)} \ar[rr]^{s^{n-1}_{D,e}(x)\ \circ^{n-1}_{D-\{e\},d}\ s^{n-1}_{D,e}(y)} \ar@{:>}[drr]|{x\ \circ^n_{D,d}\ y} & & \ar[d]^{t^{n-1}_{D,d}(x).} 
\\
\ar[rr]_{t^{n-1}_{D,e}(x)\ \circ^{n-1}_{D-\{e\},d}\ t^{n-1}_{D,e}(y)} & & 
}}
\end{equation*}
\item 
\emph{nullary reflectors}
\begin{equation*} 
\iota^n_{D,d}: \Qs^{n-1}_{D-\{d\}}\to \Qs^n_{D}, 
\quad \forall n\in \NN_0, \quad \forall d\in D\subset \NN_0\st |D|=n, 
\end{equation*}
where the following structural axioms are assumed: 
\begin{gather*}
s^{n-1}_{D,d} (\iota^n_{D,d}(x))=x=t^{n-1}_{D,d} (\iota^n_{D,d}(x)), 
\\
s^{n-1}_{D,e} (\iota^n_{D,d}(x))=\iota^{n-1}_{D,d}(s^{n-2}_{D-\{d\},e}(x)),
\quad \quad 
t^{n-1}_{D,e} (\iota^n_{D,d}(x))=\iota^{n-1}_{D,d}(t^{n-2}_{D-\{d\},e}(x)), \quad \quad \forall e\neq d. 
\end{gather*}
\begin{equation*}
\vcenter{\xymatrix{
{a} \ar[d]_{x} 
\\
{b} }
}
\quad 
\mapsto
\quad
\vcenter{\xymatrix{
\ar@{.>}[rr]^{\iota^{n-1}_{D,d}(a)} \ar[d]_{x} \ar@{:>}[drr]|{\iota^n_{D,d}(x)} & & \ar[d]^{x} 
\\
\ar@{.>}[rr]_{\iota^{n-1}_{D,d}(b)} & & 
}}, \quad \text{where} \quad a:=s^{n-2}_{D-\{d\},e}(x), \ b:=t^{n-2}_{D-\{d\},e}(x). 
\end{equation*}
\item
\emph{unary self-dualities}
\begin{equation*}
*^n_{D,d}:\Qs^n_D\to \Qs^n_D, 
\quad \forall n\in \NN_0 \quad \forall d\in D\subset \NN_0 \st |D|=n, 
\end{equation*}
where we assume the following structural axioms: 
\begin{gather*}
s^{n-1}_{D,e}(x^{*^n_{D,d}})=(s^{n-1}_{D,e}(x))^{*^{n-1}_{D-\{e\},d}}, \quad \quad 
t^{n-1}_{D,e}(x^{*^n_{D,d}})=(t^{n-1}_{D,e}(x))^{*^{n-1}_{D-\{e\},d}}, \quad \quad \forall e\neq d, 
\\
s^{n-1}_{D,d}(x^{*^n_{D,d}})=t^{n-1}_{D,d}(x),
\quad\quad
t^{n-1}_{D,d}(x^{*^n_{D,d}})=s^{n-1}_{D,d}(x). 
\end{gather*}
\begin{equation*}
\vcenter{\xymatrix{
\ar@{:>}[drr]|{x} \ar[rr]^{s_{D,e}^{n-1}(x)} \ar[d]_{s_{D,d}^{n-1}(x)} & & \ar[d]^{t_{D,d}^{n-1}(x)} 
\\
\ar[rr]_{t_{D,e}^{n-1}(x)} & & 
}
}
\quad 
\mapsto
\quad 
\vcenter{\xymatrix{
\ar[d]_{s_{D,d}^{n-1}(x)} & & \ar@{:>}[dll]|{x^{*^n_{D,d}}} \ar[ll]_{(s_{D,e}^{n-1}(x))^{*^{n-1}_{D-\{e\},d}}} \ar[d]^{t_{D,d}^{n-1}(x).}
\\ 
& & \ar[ll]^{(t_{D,e}^{n-1}(x))^{*^{n-1}_{D-\{e\},d}}}
}}
\end{equation*}
\end{itemize}
A \emph{reflective cubical $\omega$-set} is a cubical $\omega$-set equipped with the reflectors as above; a \emph{self-dual cubical $\omega$-set} is a cubical $\omega$-set equipped with the previous self-dualities. 
A \emph{cubical $\omega$-magma} is a cubical $\omega$-set equipped with the above defined binary compositions; a \emph{reflective (self-dual) cubical $\omega$-magma} is a cubical $\omega$-set equipped with reflectors (self-dualities) and compositions. 

\medskip 

A \emph{morphism of reflective cubical $\omega$-sets} is a morphism $(\phi^n_D)_{n\in\NN, \ D\subset\NN_0 \st |D|=n}$ of cubical $\omega$-sets that also satisfies: 
$\phi^n_D\circ\iota^n_{D,d}=\hat{\iota}^n_{D,d}\circ\phi^{n-1}_{D-\{d\}}$, for all $n\in\NN_0$, $D\subset \NN_0$ with $|D|=n$, $d\in D$.
 
\medskip 

A \emph{morphism of self-dual cubical $\omega$-sets} is a morphism $(\phi^n_D)_{n\in\NN, \ D\subset\NN_0 \st |D|=n}$ of cubical $\omega$-sets that also satisfies: 
$\phi^n_D\circ *^n_{D,d}=\hat{*}^n_{D,d}\circ\phi^{n}_{D}$, for all $n\in\NN$, $D\subset \NN_0$ with $|D|=n$, $d\in D$.

\medskip 

A \emph{morphism of cubical $\omega$-magmas} is a morphism $(\phi^n_D)_{n\in\NN, \ D\subset\NN_0 \st |D|=n}$ of cubical $\omega$-sets that also satisfies: 
$\phi^n_D (x\circ^n_{D,d}y)=\phi^n_D(x)\hat{\circ}^n_{D,d}\phi^{n}_{D}(y)$, for all $n\in\NN_0$, $D\subset \NN_0$ with $|D|=n$, $d\in D$ and $(x,y)\in\Qs^n_D\times_{\Qs^{n-1}_D-\{d\}}\Qs^n_D$.
\end{definition}

To obtain strict cubical $\omega$-categories we further impose the usual algebraic axioms. 
\begin{definition}
A \emph{strict cubical $\omega$-category} is a cubical reflective $\omega$-magma such that the following algebraic axioms are satisfied: 
\begin{itemize}
\item
\emph{associativity of compositions}: for all $n\in \NN_0$, for all $D\subset \NN_0$ with $|D|=n$ and for all $d\in D$: 
\begin{equation*}
x\circ^n_{D,d} (y \circ^n_{D,d} z) = (x\circ^n_{D,d} y)  \circ^n_{D,d} z, 
\quad \quad \forall (x,y,z)\in \Qs^n_D\times_{\Qs^{n-1}_{D-\{d\}}} \Qs^n_D\times_{\Qs^{n-1}_{D-\{d\}}} \Qs^n_D, 
\end{equation*}
\item 
\emph{unitality of compositions}: for all $n\in \NN_0$, for all $D\subset \NN_0$ with $|D|=n$ and for all $d\in D$: 
\begin{equation*}
x\circ^n_{D,d}\iota^n_{D,d}(s^{n-1}_{D,d} (x))=x=\iota^n_{D,d} (t^{n-1}_{D,d} (x))\circ^n_{D,d} x, \quad \forall x\in \Qs^n_D,
\end{equation*}
\item 
\emph{functoriality of identities}: for all $n\in \NN_0-\{1\}$, for all $D\subset \NN_0$ with $|D|=n$ and for all $e\neq d \in D$:  
\begin{equation*}
\iota^n_{D,d} (x\circ^{n-1}_{D-\{d\},e} y)=\iota^n_{D,d} (x)\circ^n_{D,e} \iota^n_{D,d} (y), \quad \quad \forall (x,y)\in \Qs^{n-1}_D\times_{\Qs^{n-2}_{D-\{d\}}} \Qs^{n-1}_D,
\end{equation*}
\item
\emph{exchange property}: for all $n\in \NN_0$, for all $D\subset \NN_0$ with $|D|=n$ and for all $e\neq f \in D$: 
\begin{equation*}
(x\circ^n_{D,e} y) \circ^n_{D,f} (w\circ^n_{D,e} z) = (x\circ^n_{D,f} w) \circ^n_{D,e} (y\circ^n_{D,f} z), 
\quad \quad \forall (x,y),(w,x)\in \Qs^n_D\times_{\Qs^{n-1}_{D-\{e\}}} \Qs^n_D, \quad 
  (x,w),(y,z)\in \Qs^n_D\times_{\Qs^{n-1}_{D-\{f\}}} \Qs^n_D. 
\end{equation*}
\end{itemize}
A \emph{covariant functor between cubical $\omega$-categories} is just a morphism of reflective cubical $\omega$-magmas. 
\end{definition}

\begin{definition}\label{def: axinv}
A \emph{strict involutive cubical $\omega$-category} further requires these algebraic axioms:
\begin{itemize}
\item 
\emph{involutivity}: for all $\in \NN_0$, for all $D\subset\NN_0$ with $|D|=n$ and $d\in D$, 
\begin{equation*}
(x^{*^n_{D,d}})^{*^n_{D,d}}=x, 
\quad \quad \forall x\in \Qs^n_D, 
\end{equation*}
\item 
\emph{commutativity of involutions}: for all $\in \NN_0$, for all $D\subset\NN_0$ with $|D|=n$,
\begin{equation*}
(x^{*^n_{D,e}})^{*^n_{D,f}}=(x^{*^n_{D,f}})^{*^n_{D,e}}, 
\quad \quad \forall x\in \Qs^n_D, \quad \forall e\neq f\in D, 
\end{equation*}
\item
\emph{functoriality of involutions} for all $\in \NN_0$, for all $D\subset\NN_0$ with $|D|=n$,
\begin{gather*}
(x\circ^n_{D,d} y)^{*^n_{D,d}}=(y^{*^n_{D,d}})\circ^n_{D,d} (x^{*^n_{D,d}}), 
\quad \quad \forall d\in D, 
\\
(x\circ^n_{D,d} y)^{*^n_{D,e}}=(x^{*^n_{D,e}})\circ^n_{D,d} (y^{*^n_{D,e}}), 
\quad \quad \forall d\neq e\in D,  
\end{gather*}
\item
\emph{Hermitianity of identities}: for all $\in \NN_0$, for all $D\subset\NN_0$ with $|D|=n$,
\begin{gather*}
(\iota^n_{D,d}(x))^{*^n_{D,d}}=\iota^n_{D,d}(x), \forall x\in \Qs^n_D
\\
(\iota^n_{D,d}(x))^{*^n_{D,e}}=\iota^n_{D,d}(x^{*^n_{D,e}}), 
\quad \quad \forall x\in \Qs^n_D, \quad \quad \forall d\neq e\in D
\end{gather*}
\end{itemize}
A \emph{covariant functor between involutive cubical $\omega$-categories} is a morphism of self-dual reflective cubical $\omega$-magmas. 
\end{definition}

\section{Penon Kachour Weak (Involutive) Cubical $\omega$-categories}\label{sec: weak}

We proceed to define Penon-Kachour contractions in the cubical setting. 

\begin{definition}\label{def: contraction}
Given a cubical (self-dual) reflective $\omega$-magma $\Ms$, a strict cubical (involutive) $\omega$-category $\Cs$ and a morphism of cubical (self-dual) reflective $\omega$-magmas $\Ms\xrightarrow{\pi}\Cs$, a \emph{Penon-Kachour $\pi$-contraction} is a family of maps 
$\kappa^n_{D,d}: 
\Ms^{n-1}_D(\pi)\to\Ms^{n}_{D\cup\{d\}}$, for all $n\in\NN_0$, $D\subset\NN_0$ with $|D|=n$ and all $d\in\NN_0-D$ such that: 
\begin{gather*}
\Ms^{n-1}_D(\pi):=\left\{ (x,y)\in\Ms^{n-1}_D\times\Ms^{n-1}_D \ | \ \pi(x)=\pi(y)\right\}, 
\\ 
s^{n-1}_{D\cup\{d\},d}(\kappa^n_{D,d}(x,y))=x, \quad t^{n-1}_{D\cup\{d\},d}(\kappa^n_{D,d}(x,y))=y,
\\
s^{n-1}_{D\cup\{d\},e}(\kappa^n_{D,d}(x,y))=
\kappa^{n-1}_{D-\{e\},d}\left(s^{n-2}_{D,e}(x),
s^{n-2}_{D,e}(y)\right), 
\quad 
t^{n-1}_{D\cup\{d\},e}(\kappa^n_{D,d}(x,y))=
\kappa^{n-1}_{D-\{e\},d}\left(t^{n-2}_{D,e}(x),
t^{n-2}_{D,e}(y)\right), 
\ \forall e\in D, 
\\
\pi^{n}_{D\cup\{d\}}(\kappa^n_{D,d}(x,y))=\iota^{n}_{D\cup{d},d}(\pi^{n-1}_{D}(x)) =\iota^{n}_{D\cup{d},d}(\pi^{n-1}_{D}(y)), 
\\
x=y\in\Ms^{n-1}_D \imp \kappa^n_{D,d}(x,y)=\iota^n_{D,d}(x),
\end{gather*}
\begin{equation*}
\xymatrix{
\ar@<4ex>[rr]^x \ar@<2ex>[rr]_y & \ar@{|->}[d]_{\pi} & 
\\
\ar@<-1ex>[rr]_{\pi(x)=\pi(y)}  & & 
}
\quad \quad  \quad \quad 
\xymatrix{
\ar[rr]^x \ar@{:>}[drr]|{\kappa^n_{D,d}(x,y)}\ar[d]_{\kappa^{n-1}_{D-\{e\},d}(s^{n-2}_{D,e}(x),s^{n-2}_{D,e}(y))}  & & \ar[d]^{\kappa^{n-1}_{D-\{e\},d}(t^{n-2}_{D,e}(x),t^{n-2}_{D,e}(y)).} 
\\ 
\ar[rr]_y & & 
\\
& & 
}
\end{equation*}
A \emph{morphism of cubical Penon-Kachour contractions} $(\Ms\xrightarrow{\pi}\Cs,\kappa)\xrightarrow{(\phi,\Phi)} (\hat{\Ms}\xrightarrow{\hat{\pi}}\hat{\Cs},\hat{\kappa})$ is given by a covariant morphism of reflexive (self-dual) $\omega$-magmas $\Ms\xrightarrow{\Phi}\hat{\Ms}$, a covariant (involutive) functor $\Cs\xrightarrow{\phi}\hat{\Cs}$ such that: 
\begin{gather*}
\hat{\pi}\circ\Phi=\phi\circ\pi, 
\quad\quad 
\Phi\circ\kappa=\hat{\kappa}\circ\phi. 
\end{gather*}
\end{definition}

With some abuse of notation, we denote by $\Ug$ forgetful functors, without explicitly indicating the categories (that will be clear from the context). 

\begin{definition}\label{def: free}
A \emph{free (self-dual, reflective) cubical $\omega$-magma over a cubical $\omega$-set $\Qs$} is a morphism of cubical $\omega$-sets $\Qs\xrightarrow{\eta}\Ug(\Ms(\Qs))$, into a (self-dual, reflective) cubical $\omega$-magma $\Ms(\Qs)$, such that the following universal factorization property holds: for any other morphism of cubical $\omega$-sets $\Qs\xrightarrow{\phi}\Ug(\Ms)$ into another (self-dual, reflective) cubical $\omega$-magma, there exists a unique morphism of (self-dual, reflective) $\omega$-magmas $\Ms(\Qs)\xrightarrow{\hat{\phi}}\Ms$ such that $\phi=\Ug(\hat{\phi})\circ\eta$. 

\medskip 

A \emph{free (involutive) cubical $\omega$-category over a cubical $\omega$-set $\Qs$} is a morphism of cubical $\omega$-sets $\Qs\xrightarrow{\eta}\Ug(\Cs(\Qs))$, into an (involutive) cubical $\omega$-category $\Cs(\Qs)$, such that the following universal factorization property holds: for any other morphism of cubical $\omega$-sets $\Qs\xrightarrow{\phi}\Us(\Cs)$ into another (involutive) cubical $\omega$-category, there exists a unique morphism of (involutive) $\omega$-categories $\Cs(\Qs)\xrightarrow{\hat{\phi}}\Cs$ such that $\phi=\Ug(\hat{\phi})\circ\eta$. 

\medskip 

A \emph{free (self-dual) cubical Penon-Kachour $\omega$-contraction over a cubical $\omega$-set $\Qs$} is a morphism of cubical $\omega$-sets $\Qs\xrightarrow{\eta}\Ug(\Ms)$ into the underlying cubical $\omega$-set $\Ug(\Ms)$ of the magma of a (self-dual) Penon-Kachour contraction $(\Ms\xrightarrow{\pi}\Cs,\kappa)$, such that the following universal factorization property holds: for any other morphism $\Qs\xrightarrow{\phi}\Ug(\hat{\Ms})$ of cubical $\omega$-sets  into the underlying cubical $\omega$-set $\Ug(\hat{\Ms})$ of the magma of another (self-dual) Penon-Kachour contraction $(\hat{\Ms}\xrightarrow{\hat{\pi}}\hat{\Cs},\hat{\kappa})$, there exists a unique morphism of (self-dual) Penon-Kachour contractions 
$(\Ms\xrightarrow{\pi}\Cs,\kappa)\xrightarrow{(\hat{\phi},\hat{\Phi})} (\hat{\Ms}\xrightarrow{\hat{\pi}}\hat{\Cs},\hat{\kappa})$ such that $\phi=\Ug((\hat{\phi},\hat{\Phi}))\circ\eta$. 
\end{definition}

The uniqueness of free structures, up to a unique isomorphism compatible with the universal factorization property, is assured from the definition. The existence is proved in lemma~\ref{lem: free-magma} below. 

\begin{lemma}\label{lem: free-magma}
There exists a free self-dual reflective cubical $\omega$-magma over a cubical $\omega$-set $\Qs$. 
\end{lemma}
\begin{proof}
The following proof follows the recursive construction strategy in~\cite[proposition~3.1]{BeBe17}, also recalled in~\cite[proposition~3.2 point a.]{BeBe23}, adapted to our specific cubical $\omega$-set definition. 

\medskip 

We start with a given cubical $\omega$-set $\left( \Qs^{n}_{D-\{d\}} \xleftarrow{s^n_{D,d}, \ t^n_{D,d}} \Qs^{n+1}_{D}\right)$, with $n\in\NN$, $D\subset \NN_0$ such that $|D|=n$ and $d\in D$. 

\medskip 

We are going to construct a self-dual reflective cubical $\omega$-magma $\left( \Ms(\Qs)^{n}_{D-\{d\}} \xleftarrow{\hat{s}^n_{D,d}, \ \hat{t}^n_{D,d}} \Ms(\Qs)^{n+1}_{D}\right)$, with compositions $\circ^n_{D,d}$, self-dualities $*^n_{D,d}$ and reflectors $\iota^n_{D,d}$ as in definition~\ref{def: operations}; and a morphism of cubical $\omega$-sets $\left(\Qs^n_D\xrightarrow{\eta^n_D}\Ms(\Qs)^n_D\right)$ that satisfies the universal factorization property in the first part of definition~\ref{def: free}.  

\medskip 

We start, for $n:=0$ and necessarily $D:=\varnothing$, defining $\Ms(\Qs)_D^0:=\Qs_D^0$ and $\Qs_D^0\xrightarrow{\eta^0_D} \Ms(\Qs)_D^0$ as the identity map. 

\medskip 

The construction of ``free 1-arrows'' starts defining free 1-identities, in every direction $D:=\{d\}$ with $d\in\NN_0$, corresponding to the already available objects in $\Ms(\Qs)^0_\varnothing$: we set, for all $d\in\NN_0$ and 1-direction $D:=\{d\}$, $d(\Qs^0):=\{(x,d) \ | \ x\in\Qs^0\}$ and $\Ms(\Qs)^1[0]^0_D:=\Qs^1_D\cup d(\Qs^0)$; furthermore we extend the definition of sources and targets for the extra identity 1-arrows: $\Ms(\Qs)^0_\varnothing\xleftarrow{s^0_{D,d},\ t^0_{D,d}} d(\Qs^0)$ by $s^0_{D,d}(x,d):=x=:t^0_{D,d}(x,d)$. 

\medskip 

We also introduce the structural map $\eta^1_D:\Qs^1_D\to \Ms(\Qs)^1[0]^0_D$ as the inclusion of $\Qs^1_D$.

\medskip 

We now further introduce arbitrary free duals (in the already available direction) of the 1-arrows in $\Ms(\Qs)^1[0]^0$ by the following iterative procedure: suppose that $\Ms(\Qs)^1[0]^j$ has been already constructed; \footnote{
Notice that the running index $j\in\NN$ is here denoting the number of successive iterations of a given duality, here denoted by the symbol $\gamma_d$, applied to an element $x\in\Ms(\Qs)^1[0]^0$. 
} 
for all $d\in\NN_0$ and $D:=\{d\}$ we provide $\Ms(\Qs)^1[0]^{j+1}_D:=\{(x,\gamma_d) \ | \  x\in\Ms(\Qs)^1[0]^j_D\}$; furthermore, we extend the source and target maps to the new extra free dual 1-arrows: $s^0_{D,d}(x,\gamma_d):=t^0_{D,d}(x)$ and $t^0_{D,d}(x,\gamma_d):=s^0_{D,d}(x)$, for all $x\in\Ms(\Qs)^1[0]^j_D$ and $D=\{d\}$ with $d\in\NN_0$. We then take $\Ms(\Qs)^1[0]_D:=\bigcup_{j\in\NN}\Ms(\Qs)^1[0]^j_D$ with the given source and targets. 

\medskip 

The next step consists in introducing free ``concatenations'' (in the only available direction) of the previous 1-arrows (and their source/target maps). Suppose that we already got $\Ms(\Qs)^1[m]$ for all $0\leq m\leq k$; for all $d\in\NN_0$, $D:=\{d\}$, we recursively introduce: \footnote{Notice that here the running index $m\in\NN_0$ denotes the level of concatenations, corresponding to the number of compositions in the given direction $d$.} 
\begin{equation*}
\Ms(\Qs)^1[k+1]^0_D:=\left\{(x,d,y) \ | \ (x,y)\in\Ms(\Qs)^1[i]_D\times\Ms(\Qs)^1[j]_D, \ i+j=k+1, \ s^0_{D,d}(x)=s^0_{D,d}(y)\right\}; 
\end{equation*}
we also recursively extend the source and target maps to the newly introduced free concatenations: 
\begin{equation*}
s^0_{D,d}(x,d,y):=s^0_{D,d}(y), \quad \quad t^0_{D,d}(x,d,y):=t^0_{D,d}(x), \quad\quad \forall (x,d,y)\in\Ms(\Qs)^1[k+1]_D. 
\end{equation*}
The family $\Ms(\Qs)^1[k+1]_D:=\bigcup_{j\in\NN}\Ms(\Qs)^1[k+1]^j_D$, for $D:=\{d\}$ and $d\in\NN_0$, with its source and target maps into $\Ms(\Qs)^0$, is obtained repeating the iteration construction of duals. 

\medskip 

Then we introduce $\Ms(\Qs)^1_D:=\bigcup_{k\in\NN}\Ms(\Qs)^1[k]_D$ with the already disjointly defined sources and targets. 

\bigskip 

As final recursive step, suppose now that we already defined $\Qs^n_{D'}\xrightarrow{\eta^n_{D'}}\Ms(\Qs)^n_{D'}$, for $D'\subset\NN_0$ with $|D'|=n$, and, for all $d\in D'$, also all the source and target maps $\Ms(\Qs)^{n-1}_{D'-\{d\}}\xleftarrow{s^{n-1}_{D',d}, \ t^{n-1}_{D',d}}\Ms(\Qs)^n_{D'}$, we proceed to define the next stage $\Ms(\Qs)^n_{D-\{d\}}\xleftarrow{s^{n}_{D,d}, \ t^{n}_{D,d}}\Ms(\Qs)^{n+1}_D$, for all $D\subset \NN_0$ with $|D|=n+1$ and $d\in D$, with the structural maps $\eta^{n+1}_D:\Qs^{n+1}_D\to\Ms(\Qs)^{n+1}_D$. 

\medskip 

We start setting $\Ms(\Qs)^{n+1}[0]^0_D:=\Qs^{n+1}_D\cup \left(\bigcup_{d\in D} d(\Qs^n_{D-\{d\}})\right)$, where, $d(\Qs^n_{D-\{d\}}):=\{(x,d) \ | \ x\in \Qs^n_{D-\{d\}}\}$, for all $D\subset \NN_0$ with $|D|=n+1$ and $d\in D$. We also extend the source and target maps to each set $d(\Qs^n_{D-\{d\}})$, for $d\in D$,
via $s^n_{D,d}(x,d):=x=:t^n_{D,d}(x,d)$ and, whenever $e\neq d\in D$, with $s^n_{D,e}(x,d)=(s^{n-1}_{D-\{d\},e}(x),e)$, 
$t^n_{D,e}(x,d)=(t^{n-1}_{D-\{d\},e}(x),e)$. 

\medskip 

Then we recursively introduce $\Ms(\Qs)^{n+1}[0]^{j+1}_D:=\{(x,\gamma_d) \ | \ x\in\Ms(\Qs)^{n+1}[0]^j_D, \ d\in D \}$; we further extend the source and target maps as $s^n_{D,d}(x,\gamma_d):=t^n_{D,d}(x)$,  $t^n_{D,d}(x,\gamma_d):=s^n_{D,d}(x)$ and, whenever $d\neq e\in D$, via 
$s^n_{D,e}(x,\gamma_d):=s^n_{D,e}(x)$ and $t^n_{D,e}(x,\gamma_d):=t^n_{D,e}(x)$; 
finally we set $\Ms(\Qs)^{n+1}[0]_D:=\bigcup_{j\in\NN}\Ms(\Qs)^{n+1}[0]^j_D$, for all $D\subset \NN_0$ with $|D|=n+1$ with the already introduced source and target maps. 

\medskip 

At last we suppose already defined all $\Ms(\Qs)^{n+1}[m]_D$, for all $0\leq m\leq k$, with their source and target maps and we are going to introduce 
\begin{equation*}
\Ms(\Qs)^{n+1}[k+1]^0_D:=\left\{(x,d,y) \ | \ (x,y)\in\Ms(\Qs)^{n+1}[i]_D\times\Ms(\Qs)^{n+1}[j]_D, \ i+j=k+1, \ d\in D, \ s^n_{D,d}(x)=t^n_{D,d}(y)\right\}
\end{equation*}
defining $s^n_{D,d}(x,d,y)=s^n_{D,d}(y)$, $t^n_{D,d}(x,d,y)=t^n_{D,d}(x)$ and, whenever $e\neq d\in D$, 
$s^n_{D,e}(x,d,y)=(s^n_{D,e}(x),d,s^n_{D,e}(y))$ 
$t^n_{D,d}(x,d,y)=(t^n_{D,e}(x),d,t^n_{D,e}(y))$; setting $\Ms(\Qs)^{n+1}[k]_D:=\bigcup_{j\in\NN_0}\Ms(\Qs)^{n+1}[k]^j_D$, with the same previous recursion strategy freely adding dual $(n+1)$-arrows, we finally define $\Ms(\Qs)^{n+1}_D:=\bigcup_{k\in\NN}\Ms(\Qs)^{n+1}[k]_D$, with its already locally well-defined source and target maps. 

\medskip 

We also define $\eta^{n+1}_D:\Qs^{n+1}_D\to \Ms(\Qs)^{n+1}_D$ as the inclusion into $\Ms(\Qs)^{n+1}[0]^0_D\subset \Ms(\Qs)^{n+1}_D$. 

\bigskip 

Up to this point we managed to recursively define a morphism $\Qs\xrightarrow{\eta}\Ms(\Qs)$ of cubical $\omega$-sets.  

\medskip 

The nullary, unary and binary operations on the cubical $\omega$-set $\Ms(\Qs)$ are readily available as follows: 
\begin{gather*}
\iota^n_{D,d}:\Ms(\Qs)^{n-1}_{D-\{d\}}\to \Ms(\Qs)^n_D, \quad \quad x\mapsto (x,d), 
\\
*^n_{D,d}: \Ms(\Qs)^{n}_{D}\to \Ms(\Qs)^n_D, \quad \quad(x)^{*^n_{D,d}}:=(x,\gamma_d), 
\\
\circ^n_{D,d}:\Ms(\Qs)^n_D\times_{\Ms{\Qs}^{n-1}_{D-\{d\}}}\Ms(\Qs)^n_D\to\Ms(\Qs)^n_D, \quad \quad 
(x\circ^n_{D,d}y):=(x,d,y). 
\end{gather*}
With such definition and the already provided recursive definition of source and target maps, the cubical $\omega$-set $\Ms(\Qs)$ becomes a self-dual reflective cubical $\omega$-magma. 

\bigskip 

We only need to check the universal factorization property of the morphism $\Qs\xrightarrow{\eta}\Ms(\Qs)$. 

Given a morphism $\Qs\xrightarrow{\phi}\Ms$ into the underlying cubical $\omega$-set of a self-dual reflective cubical $\omega$-magma $\Ms$, the requirement $\phi=\hat{\phi}\circ\eta$ already implies that the restriction of $\hat{\phi}$ to the cubical $\omega$-subset $\Qs$ must coincide with $\phi$. Since $\Ms(\Qs)\xrightarrow{\hat{\phi}}\Ms$ must be a morphism of self-dual reflective cubical $\omega$-magmas, we necessarily have $\hat{\phi}(\iota^{n+1}_{D,d}(x))=\iota^{n+1}_{D,d}(\hat{\phi}^n_D(x))$, hence $(x,d)\mapsto (\phi(x),d)$; similarly $\hat{\phi}(x^{*^n_{D,d}})=(\hat{\phi}(x))^{*^n_{D,d}}$ and finally $\hat{\phi}(x\circ^n_{D,d}y)=\hat{\phi}(x)\circ^n_{D,d}\hat{\phi}(y)$ and hence the morphism $\hat{\phi}$ is uniquely determined by our recursive construction, once it has been fixed (as in this case) on $\eta(\Qs)$. 
\end{proof}

Instead of giving a direct recursive proof, the following lemma~\ref{lem: free-cat} is obtained with the same ``quotient by congruences'' technique as in~\cite[section 3.2]{BeBe17}. In order to do so, we briefly recall the necessary preliminary material on congruences in the present setting of cubical $\omega$-magmas: 
\begin{itemize}
\item 
The category of morphisms of cubical $\omega$-sets/magmas admits finite products (it is actually complete). Given two cubical $\omega$-magmas $\Ms,\Ns$, their \emph{product $\omega$-magma} $\Ms\times\Ns$ can be constructed via Cartesian products $(\Ms\times\Ns)^n_D:=\Ms^n_D\times\Ns^n_D$, for $n\in\NN$ and $D\subset\NN_0$ with $|D|=n$, equipped with componentwise defined sources/target maps, reflectors, self-dualities and compositions. 
\item
A \emph{congruence $\Rs$ in a cubical $\omega$-magma $\Ms$} is a cubical $\omega$-magma $\Rs$ such that $\Rs^n_D\subset \Ms^n_D\times\Ms^n_D$, for all $n\in\NN$ and all $D\subset \NN_0$ with $|D|=n$, and such that the inclusion $\left(\Rs^n_D\xrightarrow{\nu^n_D}\Ms^n_D\times\Ms^n_D\right)$ is a morphism of cubical $\omega$-magmas, from $\Rs$ into the product cubical $\omega$-magma $\Ms\times\Ms$. \footnote{
Equivalently $\Rs$ is a cubical $\omega$-subset of the product cubical $\omega$-set $\Ms\times\Ms$ that is algebraically closed under all the nullary reflectors, unary self-dualities and binary composition operations in the cubical $\omega$-magma $\Ms$. 
}  
\item 
Given a congruence $\Rs$ in a cubical $\omega$-magma $\Ms$, we define the \emph{quotient $\omega$-magma $\Ms/\Rs$} and the \emph{quotient morphism $\left(\Ms^n_D\xrightarrow{\pi^n_D}(\Ms/\Rs)^n_D\right)$}, for $n\in\NN$, $D\subset \NN_0$ with $|D|=n$, as follows: 
\begin{itemize}
\item[] 
the quotient sets $(\Ms/\Rs)^n_D:=\Ms^n_D/\Rs^n_D$ are a cubical $\omega$-magma with well-defined sources/targets:
\begin{equation*} 
[x]_{\Rs^n_D}\mapsto [s^n_{D,d}(x)]_{\Rs^n_{D-\{d\}}}, 
\quad \quad 
[x]_{\Rs^n_D}\mapsto [t^n_{D,d}(x)]_{\Rs^n_{D-\{d\}}}, 
\quad \quad 
\forall x\in\Ms^{n+1}_D, \quad d\in D;
\end{equation*} 
and one gets a (self-dual reflective) cubical $\omega$-magma with the well-defined operations: 
\begin{equation*}
[x]_{\Rs^n_D}\hat{\circ}^n_{D-\{d\}}[y]_{\Rs^n_D}:=[x\circ^n_{D-\{d\}}y]_{\Rs^n_D}, 
\quad
([x]_{\Rs^n_D})^{\hat{*}^n_{D,d}}:=[x^{*^n_{D-\{d\}}}]_{\Rs^n_D}, 
\quad 
\hat{\iota}^{n+1}_{D,d}([x]_{\Rs^n_D}):=[\iota^{n+1}_{D,d}(x)]_{\Rs^{n+1}_D}, \quad \forall x,y\in\Ms^n_D. 
\end{equation*}
\item[] 
the maps $\pi^n_D: x\mapsto [x]_{\Rs^n_D}$, for $x\in\Ms^n_D$, provide the quotient morphism between cubical $\omega$-magmas. 
\end{itemize}
\item 
Every morphism $\Ms\xrightarrow{\phi}\Cs$ of self-dual reflective cubical $\omega$-magmas induces a \emph{kernel congruence} of self-dual reflective $\omega$-magmas $\Ks_\phi\subset \Ms\times\Ms$ defined by: 
\begin{equation*}
\Ks_\phi:=\{(x,y)\in\Ms\times\Ms \ | \ \phi(x)=\phi(y) \}. 
\end{equation*}
\item 
Let $\Ms\xrightarrow{\phi}\Cs$ be a morphism of self-dual reflective cubical $\omega$-magmas, given another congruence in $\Ms$ with $\Es\subset\Ks_\phi$, there exists a unique morphism $\Ms/\Es\xrightarrow{\hat{\phi}}\Cs$ of self-dual reflective cubical $\omega$-magmas such that $\phi=\hat{\phi}\circ \pi_\Es$, where $\Ms\xrightarrow{\pi_\Es}\Ms/\Es$ is the quotient morphism. 
The well-defined morphism $\hat{\phi}$ is uniquely determined by the relation $\hat{\phi}([x]_\Es):=\phi(x)$, for all $x\in\Ms$. 
\end{itemize}

\begin{lemma}\label{lem: free-cat}
There exists a free involutive cubical strict $\omega$-category over a cubical $\omega$-set $\Qs$. \footnote{
For simplicity, we omit in the following the explicit indication of the forgetful functors.
} 
\end{lemma}
\begin{proof}
Starting with the cubical $\omega$-set $\Qs$, we first utilize lemma~\ref{lem: free-magma} to produce $\Qs\xrightarrow{\eta}\Ms(\Qs)$, a free self-dual reflective cubical $\omega$-magma over $\Qs$. 

\medskip 

For $n\in\NN_0$, $D\subset\NN_0$ with $|D|=n$, consider the family of relations $\Xs^n_{D}\subset\Ms(\Qs)^n_D\times\Ms(\Qs)^n_D$ consisting of all the pairs of elements corresponding to the ``missing cubical categorical axioms equalities'' within terms of $\Ms(\Qs)$; 
in practice $\Xs^n_D$ is obtained as the union of the following families of subsets of $\Ms(\Qs)^n_D\times\Ms(\Qs)^n_D$: 
\begin{align} \notag 
\bigcup_{\phantom{e\neq}d\in D}
&\left\{\left(x\circ^n_{D,d} (y \circ^n_{D,d} z),\ (x\circ^n_{D,d} y)  \circ^n_{D,d} z\right) 
\ | \ 
(x,y,z)\in \Qs^n_D\times_{\Qs^{n-1}_{D-\{d\}}} \Qs^n_D\times_{\Qs^{n-1}_{D-\{d\}}} \Qs^n_D 
\right\},
\\ \notag
\bigcup_{\phantom{e\neq}d\in D}
&\left(\left\{\left(
x\circ^n_{D,d}\iota^n_{D,d}(s^{n-1}_{D,d} (x)), \ x \right) \ | \ x\in \Qs^n_D \right\}
\cup 
\left\{\left(x, \ \iota^n_{D,d} (t^{n-1}_{D,d} (x))\circ^n_{D,d} x\right)  \ | \ 
x\in \Qs^n_D \right\}\right),
\\ \notag
\bigcup_{e\neq d\in D}
&\left\{\left(\iota^n_{D,d} (x\circ^{n-1}_{D-\{d\},e} y), \ \iota^n_{D,d} (x)\circ^n_{D,e} \iota^n_{D,d} (y)\right)
\ | \ (x,y)\in \Qs^n_D\times_{\Qs^{n-1}_{D-\{d\}}} \Qs^n_D \right\},
\\ \notag
\bigcup_{e\neq f\in D} 
&
\left\{\left(
(x\circ^n_{D,e} y) \circ^n_{D,f} (w\circ^n_{D,e} z),\  (x\circ^n_{D,f} w) \circ^n_{D,e} (y\circ^n_{D,f} z) \right)
\ | \  
\ \begin{minipage}{3.9cm}
$(x,y),(w,x)\in \Qs^n_D\times_{\Qs^{n-1}_{D-\{e\}}} \Qs^n_D$,  
\\ 
$(x,w),(y,z)\in \Qs^n_D\times_{\Qs^{n-1}_{D-\{f\}}} \Qs^n_D$ 
\end{minipage}
\right\}, 
\\ \label{eq: X}
\bigcup_{\phantom{e\neq}d\in D} 
&\left\{
\left((x^{*^n_{D,d}})^{*^n_{D,d}}, \ x\right) 
\ | \  x\in \Qs^n_D  
\right\},
\\ \notag
\bigcup_{e\neq f\in D} 
&\left\{
\left(
(x^{*^n_{D,e}})^{*^n_{D,f}}, \ (x^{*^n_{D,f}})^{*^n_{D,e}}
\right)
x\in \Qs^n_D  
\right\},
\\ \notag
\bigcup_{\phantom{e\neq}d\in D} 
&\left\{
\left(
(x\circ^n_{D,d} y)^{*^n_{D,d}}, \ (y^{*^n_{D,d}})\circ^n_{D,d} (x^{*^n_{D,d}}) 
\right)
\ | \ 
(x,y)\in \Qs^n_{D-\{d\}} \times_{}\Qs^n_{D-\{d\}} 
\right\},
\\ \notag
\bigcup_{e\neq d\in D}
&\left\{\left(
(x\circ^n_{D,d} y)^{*^n_{D,e}},\ (x^{*^n_{D,e}})\circ^n_{D,d} (y^{*^n_{D,e}}) \right)
\ | \ 
(x,y)\in \Qs^n_{D-\{d\}} \times_{}\Qs^n_{D-\{d\}} 
\right\},
\\ \notag
\bigcup_{\phantom{e\neq}d\in D}
&\left\{
\left(
(\iota^n_{D,d}(x))^{*^n_{D,d}},\ \iota^n_{D,d}(x)
\right)
\ | \ 
x\in \Qs^n_D
\right\},
\\ \notag
\bigcup_{d\neq e\in D} 
&\left\{
\left(
(\iota^n_{D,d}(x))^{*^n_{D,e}}, \ \iota^n_{D,d}(x^{*^n_{D,e}}) 
\right)
\ | \  
x\in \Qs^n_D 
\right\}. 
\end{align}

\medskip 

The \emph{congruence $\Rs_\Xs$ generated by the cubical $\omega$-relation $\Xs$ in $\Ms(\Qs)$} is the smallest congruence in $\Ms(\Qs)$ containing $\Xs$ and is obtained taking the intersection of the family of all the congruences in $\Ms(\Qs)$ containing $\Xs$. 

\medskip 

The quotient self-dual reflective cubical $\omega$-magma $\Ms(\Qs)/\Rs_\Xs$ by the congruence $\Rs_\Xs$ turns out to be a strict involutive cubical \hbox{$\omega$-category}, since $\Xs\subset\Rs_\Xs$. 

\medskip 

The composition $\Qs\xrightarrow{\eta}\Ms(\Qs)\xrightarrow{\pi}\Ms(\Qs)/\Rs_\Xs$ of the quotient morphism of self-dual reflective cubical $\omega$-magmas $\Ms(\Qs)\xrightarrow{\pi}\Ms(\Qs)/\Rs_\Xs$ with the natural inclusion of cubical $\omega$-sets $\Qs\xrightarrow{\eta}\Ms(\Qs)$, is a morphism of cubical $\omega$-sets that satisfies the universal factorization property defining free involutive cubical $\omega$-categories: 
\\
given $\Qs\xrightarrow{\phi}\Cs$ a morphism of cubical $\omega$-sets into the underlying cubical $\omega$-set of an involutive cubical $\omega$-category $\Cs$, by the universal factorization property of the free self-dual reflective cubical $\omega$-magma $\Qs\xrightarrow{\eta}\Ms(\Qs)$, there exists a unique morphism of self-dual reflective cubical $\omega$-magmas $\Ms(\Qs)\xrightarrow{\tilde{\phi}}\Cs$ such that $\phi=\tilde{\phi}\circ\eta$. 

The kernel relation $\Ks_{\tilde{\phi}}\subset \Ms(\Qs)\times\Ms(\Qs)$, induced by the morphism $\tilde{\phi}$, is a congruence of self-dual reflective cubical $\omega$-magma and it necessarily satisfies  $\Xs\subset \Ks_{\tilde{\phi}}$ and hence $\Rs_\Xs\subset\Ks_{\tilde{\phi}}$. 
It follows that there exists a unique morphism of involutive cubical $\omega$-categories $\Ms(\Qs)/{\Rs_\Xs}\xrightarrow{\hat{\phi}}\Cs$ such that $\tilde{\phi}=\hat{\phi}\circ\pi$ and so $\phi=\tilde{\phi}\circ\eta=\hat{\phi}\circ\pi\circ\eta$. 
\end{proof}

\begin{corollary}\label{cor: monad}
There is a free $\omega$-category monad obtained by composing the free involutive $\omega$-category functor with the forgetful fuctor into the category of $\omega$-sets. 
\end{corollary}

The subsequent lemma is obtained recursively, as done for the globular case in~\cite[proposition 3.3]{BeBe17}, introducing an intermediate construction of ``free cubical contraction $n$-cells'' at each stage $n$ of the construction of free self-dual reflective magmas and their quotient free involutive categories over a given cubical $\omega$-set.
\begin{lemma}\label{lem: contraction} 
There exists a free self-dual cubical Penon-Kachour contraction over a cubical $\omega$-set $\Qs$.
\end{lemma}
\begin{proof}
Starting with a cubical $\omega$-set $\Qs$, we will recursively construct a free self-dual cubical Penon-Kachour contraction $(\Ms^\kappa(\Qs)\xrightarrow{\pi}\Cs^\kappa(\Qs),\kappa)$ over $\Qs$. Notice that the self-dual relfective cubical $\omega$-magma $\Ms^\kappa(\Qs)$ and the involutive cubical $\omega$-category $\Cs^\kappa(\Qs)$ differ from the free cubical $\omega$-magma $\Ms(\Qs)$ and the free involutive cubical $\omega$-category $\Cs(\Qs)$ already introduced in lemmata~\ref{lem: free-magma} and~\ref{lem: free-cat}, since further ``free-contraction $n$-cells'' (and consequently further congruence terms) are introduced at every level $n\in\NN$ of the procedure. 

\medskip 

For $n=0$, we define $\Ms^\kappa(\Qs)^0:=\Qs^0$; we consider the empty relation $\Xs^0:=\varnothing\subset\Qs^0\times\Qs^0$ and its generated equivalence relation  $\Rs_\Xs^0=\Delta_{\Qs^0}$ (the identity equivalence relation in $\Qs^0$), obtaining $\Cs^\kappa(\Qs)^0:=\Ms^\kappa(\Qs)^0/\Rs_\Xs^0$ and the bijective quotient map $\Ms^\kappa(\Qs)^0\xrightarrow{\pi^0}\Cs^\kappa(\Qs)^0$. 
There are no object-valued free-contractions in $\Ms^\kappa(\Qs)^0$. 
The structural inclusion $\Qs^0\xrightarrow{\eta^0}\Ms^\kappa(\Qs)^0$ is just the identity map. 

\medskip 

Passing now to the case $n=1$, in principle, we should modify the construction in lemma~\ref{lem: free-magma} of the ``level-1'' free self-dual reflective magma $\Ms(\Qs)^1$, introducing as input (for the arbitrary composition of self-dualities and concatenations) not only all the 1-cells in $\Qs^1$ and the free identities $\cup_{d\in\NN_0}d(\Qs^0)$, but also the free 1-cells $\kappa^1(\pi^0)$ coming from the contractions induced by the map $\pi^0$. 

\medskip 

Since $\pi^0$ is bijective, we have $\Ms^\kappa(\Qs)(\pi)^0:=\{(x,y)\in\Ms^\kappa(\Qs)^0\times\Ms^\kappa(\Qs)^0\ | \ \pi^0(x)=\pi^0(y)\}=\Delta_{\Qs^0}$ and hence, from the last axiom in the definition of cubical Penon-Kachour contraction $\kappa^1_{\varnothing,d}:\Ms^\kappa(\Qs)^0\to\Ms^\kappa(\Qs)^1$, we obtain  $\kappa^1_{\varnothing,d}(x,y)=\iota^1_{\varnothing,d}(x)=\iota^1_{\varnothing,d}(y)\in d(\Qs^0)$, for all $(x,y)\in\Ms^\kappa(\Qs)(\pi)^0$ and all $d\in\NN_0$. Hence, in the case $n=1$ the free-contraction cells are coinciding with the already defined free level-1 identities in $\Ms(\Qs)^1$. 
Hence we simply define $\Ms^\kappa(\Qs)^1:=\Ms(\Qs)^1$ and, taking $\Rs_{\Xs}^1$ as the equivalence relation in $\Ms(\Qs)^1$ generated by all the ``axioms'' $\Xs^1$ listed in the equations~\eqref{eq: X}, we define $\Cs^\kappa(\Qs)^1:=\Cs(\Qs)^1:=\Ms(\Qs)^1/\Rs_\Xs^1$ with $\Ms^\kappa(\Qs)^1\xrightarrow{\pi^1}\Cs^\kappa(\Qs)^1$ the quotient map and contraction $\kappa^1:\Ms^\kappa(\Qs)(\pi)^0\to \Ms^\kappa(\Qs)^1$ as $\kappa^1_{\varnothing,d}(x,y):=\iota^1_{\varnothing,d}(x)=\iota^1_{\varnothing,d}(y)$, for all $d\in\NN_0$. Finally we also define the structural free-inclusion $\Qs^1\xrightarrow{\eta^1}\Ms^\kappa(\Qs)^1=\Ms(\Qs)^1$ as in lemma~\ref{lem: free-magma}. 

\medskip 

Suppose now, by recursion, that we already constructed, for a given $n\in \NN$, a morphism of self-dual reflective cubical $n$-magmas $\Ms^\kappa(\Qs)^n\xrightarrow{\pi^n}\Cs^\kappa(\Qs)^n$ onto the involutive cubical $n$-category $\Cs^\kappa(\Qs)^n$, with cubical Penon-Kachour contraction $\Ms^\kappa(\Qs)^{n-1}(\pi^n)\xrightarrow{\kappa^n}\Ms^\kappa(\Qs)^n$ and with structural morphism of cubical $n$-sets $\Qs^n\xrightarrow{\eta^n}\Ms^\kappa(\Qs)^n$. 

The projection $\pi^n$ determines the domain set $\Ms^\kappa(\Qs)(\pi)^n:=\{(x,y)\in\Ms^\kappa(\Qs)^n\times\Ms^\kappa(\Qs)^n \ | \ \pi^n(x)=\pi^n(y)\}$ of the free-contraction $\kappa^{n+1}$. We consider, as in lemma~\ref{lem: free-magma}, the $(n+1)$-cells $\Qs^{n+1}_D\cup\left(\bigcup_{d\in D} d(\Qs^n_{D-\{d\}})\right)$ (containing already the ``freely generated'' $(n+1)$-identities) and we further add the ``freely-generated'' $(n+1)$-contractions 
$\kappa_{D,d}(\Qs^n):=\left\{[x,d,y]^{n+1}_D \ | \ (x,y)\in \Ms^\kappa(\Qs)^n_{D-\{d\}}\times \Ms^\kappa(\Qs)^n_{D-\{d\}}, \ x\neq y,\  \pi^n_{D-\{d\}}(x)=\pi^n_{D-\{d\}}(y)\right\}$, for all $D\subset \NN_0$ with $|D|=n+1$ and $d\in D$. 
In this way, we introduce $\Ms^\kappa(\Qs)^{n+1}_D[0]^0:=\Qs^{n+1}_D\cup\left(\bigcup_{d\in D} d(\Qs^n_{D-\{d\}})\right)\cup\left(\bigcup_{d\in D} \kappa_{D,d}(\Qs^n_{\phantom{D}})\right)$, extending the definition of sources and targets to the extra free-contractions as required by the axioms of Penon-Kachour contraction: 
$s^n_{D,d}([x,d,y]^{n+1}_D):=x$, $t^n_{D,d}([x,d,y]^{n+1}_D):=y$ and, for all $e\in D$ with $e\neq d$,  $s^n_{D,e}([x,d,y]^{n+1}_D):=\kappa^n_{D-\{e\},d}(s^{n-1}_{D-\{d\},e}(x),s^{n-1}_{D-\{d\},e}(y))$, 
and 
$t^n_{D,e}([x,d,y]^{n+1}_D):=\kappa^n_{D-\{e\},d}(t^{n-1}_{D-\{d\},e}(x),t^{n-1}_{D-\{d\},e}(y))$. The Penon-Kachour contraction is defined as $\kappa^{n+1}_{D,d}(x,y):=[x,d,y]^{n+1_D}$, for all $(x,y)\in \Ms^\kappa(\Qs)(\pi)^n$ with $x\neq y$ and by $\kappa^{n+1}_{D,d}(x,y):=\iota^{n+1}_{D,d}(x)=\iota^{n+1}_{D,d}(y)$, whenever $x=y$. 

The iterative construction of the sets $\Ms^\kappa(\Qs)^{n+1}[k]^j$ and $\Ms^\kappa(\Qs)^{n+1},$ and its  nullary, unary and binary operations as cubical $(n+1)$-magma, proceeds at this point exactly as in lemma~\ref{lem: free-magma}; similarly, the new binary relation $\Xs^{n+1}\subset \Ms^\kappa(\Qs)^{n+1}\times\Ms^\kappa(\Qs)^{n+1}$ is obtained using the same type of pairs, as in equation~\eqref{eq: X}, but with terms from the bigger set $\Ms^\kappa(\Qs)^{n+1}$; furthermore we set $\Cs^\kappa(\Qs)^{n+1}:=\Ms^\kappa(\Qs)^{n+1}/\Rs_\Xs^{n+1}$, where $\Rs_\Xs^{n+1}$ is the congruence relation generated by $\Xs^{n+1}$ in the cubical $(n+1)$-magma $\Ms^\kappa(\Qs)^{n+1}$ and with  $\Ms^\kappa(\Qs)^{n+1}\xrightarrow{\pi^{n+1}}\Cs^\kappa(\Qs)^{n+1}$ we denote the quotient map into the cubical involutive $(n+1)$-category $\Cs^\kappa(\Qs)^{n+1}$. 

\medskip 

Now that the recursive construction of the cubical Penon-Kachour contraction $(\Ms^\kappa(\Qs)\xrightarrow{\pi}\Cs^\kappa(\Qs),\kappa)$ has been completed, we only need to show that it satisfies the universal factorization property. 

For any morphism $\Qs\xrightarrow{\phi}\hat{\Ms}$ of cubical $\omega$-magmas into the cubical $\omega$-magma $\hat{\Ms}$ of another cubical Penon-Kachour contraction $(\hat{\Ms}\xrightarrow{\hat{\pi}}\hat{\Cs},\hat{\kappa})$, we need to show the existence of a unique morphism of Penon-Kachour contractions $(\Ms^\kappa(\Qs)\xrightarrow{\pi}\Cs^\kappa(\Qs),\kappa)\xrightarrow{(\hat{\phi},\hat{\Phi})}(\hat{\Ms}\xrightarrow{\hat{\pi}}\hat{\Cs},\hat{\kappa})$ such that $\hat{\Phi}\circ\kappa=\hat{\kappa}\circ(\phi,\phi)$. 

Since $\hat{\Phi}$ is already fixed as $\Phi(\eta(x)):=\phi(x)$ on $\eta(\Qs)\subset\Ms^\kappa(\Qs)$, and since $\hat{\Phi}$ must be a morphism of cubical self-dual reflective $\omega$-magmas compatible with the contractions $\hat{\Phi}([x,d,y]^{n+1}_D)=\hat{\kappa}^{n+1}_{D,d}(\phi^n(x),\phi^n(y))$; we see that $\hat{\Phi}^{n+1}$ is uniquely determined inductively by $\hat{\Phi}^n$ and $\phi^{n+1}$, for all $n\in\NN$.

The existence of the unique morphism $\Cs^\kappa(\Qs)\xrightarrow{\hat{\phi}}\hat{\Cs}$ of involutive cubical $\omega$-categories such that $\hat{\pi}\circ\hat{\Phi}=\hat{\phi}\circ\pi$ follows immediately from the fact that the kernel relation of $\hat{\pi}\circ\hat{\Phi}$ is a congruence of cubical $\omega$-magma in $\Ms^\kappa(\Qs)$ containing the set $\Xs$ and hence its generated congruence $\Rs_\Xs$, so that there exists a unique well-defined involutive functor  $\Cs^\kappa(\Qs)\xrightarrow{\hat{\phi}}\hat{\Cs}$ given by $\hat{\phi}([x]_{\Rs_\Xs}):=\hat{\pi}(\Phi(x))$, fo all $x\in\Ms^\kappa(\Qs)$. 
\end{proof}

\begin{theorem}\label{th: weak}
There is an adjunction $\xymatrix{\Qf\rtwocell^{F}_{U}{'}&\Kf}$, $F\dashv U$ between the category of morphisms of cubical $\omega$-sets and the category of morphisms of contractions of cubical reflective (self-dual) $\omega$-magmas, where $U$ is the forgetful functor associating to every contraction $(\Ms\xrightarrow{\pi}\Cs,\kappa)$ the underlying cubical $\omega$-set of $\Ms$ and $F$ associates to every cubical $\omega$-set $\Qs$ the free contraction as constructed above in lemma~\ref{lem: contraction}.
\end{theorem}
\begin{proof}
The existence of a left adjoint functor $F$ and an adjunction $F\dashv U$ is a standard consequence of the already proved universal factorization property for the free Penon-Kachour contraction over cubical $\omega$-sets (see for example~\cite[section~2.3 and theorem~2.3.6]{Lei14}). 
\end{proof}

As a consequence of the existence of any adjunction $F\dashv U$, with unit $\eta$ and counit $\epsilon$, we have an associated monad $(U\circ F,\eta,F\circ \epsilon\circ U)$, where the unit $\eta$ of the adjunction takes the role of the monadic unit for the monad endofunctor $U\circ F$ and the monadic multiplication $F\circ \epsilon\circ U$ is obtained from the co-unit $\epsilon$ of the adjunction (see for example~\cite[section~5.1 and lemma~5.1.3]{Ri16}). 

\medskip 

After all this preliminary work, we finally arrive at our definition of involutive weak cubical $\omega$-category. 
\begin{definition}\label{def: weak}
An \emph{involutive weak cubical $\omega$-category} is an algebra for the monad $U\circ F$ associated to the adjunction $F\dashv U$. 
\end{definition}

\subsection{Examples}\label{ex: weak}

Every weak cubical $\omega$-groupoid as already studied in~\cite{Ka22} becomes an example of weak involutive $\omega$-category, simply considering as involutions of $n$-arrows the ``directional inverses'' of the cubical $n$-arrows. 

\medskip 

As a notable special example of weak cubical $\omega$-groupoid, we can consider the weak $\omega$-groupoid of homotopies (without fixed extrema) of a topological space. 

\medskip 

Every strict involutive cubical $\omega$-category is of course an example of weak involutive cubical $\omega$-category. 

\medskip

Also in this trivial strict case, the specific definition of cubical $\omega$-sets that we have adopted in the present paper is sufficiently general to allow the usage of different classes $\Qs^n_D$, depending on the choice of the \hbox{``$n$-direction'' $D$}: for example a countable family of involutive 1-categories $(\Cs_n,s_n,t_n,\circ_n,\iota_n ,*_n)$, $n\in\NN_0$, produces a \emph{product strict involutive cubical $\omega$-category} $\Ds:=\Pi_{n\in\NN_0}\Cs_n$ specified as follows: 

\begin{itemize}
\item 
for all $n\in \NN$ and $D\subset \NN_0$ with $|D|=n$, we define: 
\begin{equation*}
\Ds^n_D:=\left\{
(x_j)_{j\in\NN_0} \ | \ \forall j\in D \st x_j\in \Cs^1_j, \ \forall j\notin D \st x_j\in \Cs^0_j
\right\},
\end{equation*}
\item 
for all $n\in \NN_0$, for all $D\subset \NN_0$ with $|D|=n$ and $d\in D$, sources and targets are defined by:
\begin{gather*}
s^{n-1}_{D,d}: (x_j)_{j\in\NN_0} 
\mapsto (\hat{x}_j)_{j\in\NN_0}, \quad \text{where} \quad 
\hat{x}_j:=
\begin{cases}
x_j\phantom{s_d()} \quad j\neq d, 
\\ 
s_d(x_j) \quad j=d, 
\end{cases}
\\
t^{n-1}_{D,d}: (x_j)_{j\in\NN_0} 
\mapsto (\tilde{x}_j)_{j\in\NN_0}, \quad \text{where} \quad 
\tilde{x}_j:=
\begin{cases}
x_j\phantom{t_d()} \quad j\neq d, 
\\ 
t_d(x_j) \quad j=d,
\end{cases}
\end{gather*}
\item 
for all $n\in \NN$, $D\subset \NN_0$ with $|D|=n$ and $d\in D$, identities are given by: 
\begin{equation*}
\iota^n_{D,d}: (x_j)_{j\in\NN_0}\mapsto 
(\bar{x}_j)_{j\in\NN_0}, \quad \text{where} \quad 
\bar{x}_j:=\begin{cases}
x_j\phantom{\iota_d()} \quad j\neq d, 
\\ 
\iota_d(x_j) \quad j=d,
\end{cases}
\end{equation*}
\item 
for all $n\in\NN_0$, $D\subset \NN_0$ with $|D|=n$, $d\in D$ composition are defined via: 
\begin{equation*}
(x_j)_{j\in\NN_0}\circ^n_{D,d} (y_j)_{j\in\NN_0}:=(z_j)_{j\in\NN_0}, \quad \text{where} \quad 
z_j:=\begin{cases}
x_j=y_j \quad j\neq d, 
\\
x_j\circ_dy_j \quad j=d, 
\end{cases}
\end{equation*}
\item 
for all $n\in \NN_0$, $D\subset \NN_0$ with $|D|=n$, $d\in D$, involutions are provided by: 
\begin{equation*}
((x_j)_{j\in\NN_0})^{*^n_{D,d}}:=(w_j)_{j\in\NN_0}, \quad \text{where}\quad 
w_j:=\begin{cases}
x_j\phantom{^{*_d}} \quad j\neq d, 
\\
x_j^{*_d} \quad j=d.
\end{cases}
\end{equation*}
\end{itemize}

\medskip 

Whenever we substitute the sequence of strict involutive 1-categories above, with a sequence of weak involutive 1-categories, one immediately obtains some non-trivial examples of weak involutive cubical $\omega$-categories (for example using as morphisms bimodules over different pairs of involutive monoids). 

\medskip 

Making full use of the material on involutions of multimodules recently developed in~\cite{BeCoPu22}, one can immediately obtain weak cubical involutive $\omega$-categories, that are analogs of the example of product cubical $\omega$-categories, by considering a family $\Os$ of objects consisting of involutive monoids and $n$-arrows in the direction $D$ as left-$D$-right-$D$-multimodules between finite families (with cardinality $D$) of the monoids in $\Os$; the compositions in the direction $d$ will consists in tensor products of multimodules over a single monoid in position $d$ and involutions will consist in duals of multimodules with respect to the involutive monoids in position $d$. 

\medskip

Interestingly, the previous ``product'' examples of strict/weak involutive cubical $\omega$-categories suggests an immediate generalization of the formalism of higher categories to the case $(\Cs_\gamma)_\Gamma$ of indexes labeled by well-ordered sets $\Gamma$ of arbitrary cardinality (beyond the countable case $\NN$); we will not pursue here such directions. 

\medskip 

A similar cubical product strict/weak involutive $\omega$-category, can actually be defined for any (countable) family of strict/weak globular involutive $n$-categories simply taking sequences $(x_j)_{j\in\NN_0}$ of globular $n$-cells. 

\medskip 

More interesting examples can be obtained considering ``higher multimodules'' as in this inductive construction: 
\begin{itemize}
\item 
as objects ($n=0$), we consider involutive monoids (or more generally involutive 1-categories) $\As,\Bs,\dots$, 
\item 
as 1-morphisms, we take all the bimodules ${}_\As\Ms_\Bs$ over the already defined objects: compositions will be the usual tensor products of bimodules ${}_\As\Ms_\Bs\otimes_{\Bs} {}_\Bs\Ns_\Cs$ and involution of 1-morphisms will be the usual notion of contragradient bimodule ${}_\Bs\hat{\Ms}_\As$, 
\item 
given a 1-morphism bimodule ${}_\As\Ms_\Bs$, and its contragradient ${}_\Bs\hat{\Ms}_\As$, one constructs their generated free involutive category $\As(\Ms)^{[1]}$ with two objects $\As,\Bs$,   
\item 
one iterates the construction with the above generated involutive categories $\As^{[1]}, \Bs^{[1]},\dots$, in place of the original involutive monoids, obtaining bimodules of level-2 and so on, $\dots$,
\item 
given a square (not necessarily commutative) diagram of the level-1 bimodules, cubical 2-arrows can be defined as level-2 multimodules over the pairs of level-1 bimodules of the diagram,
\item 
proceeding recursively, given an $n$-dimensional cubical diagram of level-$(n-1)$ multimodules, one can introduce $n$-arrows as level-$n$ multimodules with $n$-source and $n$-targets consisting of the level-$(n-1)$ multimodules appearing in the diagram,
\item
the operations of composition are iterated as tensor products of level-$n$ multimodules over the involutive categories generated by level-$(n-1)$ multimodules and involutions are provided by the controgradient construction.
\end{itemize} 

\section{Outlook}\label{sec: outlook}

The present paper is only a starting point in the study of involutions suitable for the definition of operator algebraic structures in the weak infinite vertically categorified (cubical) case (see the introduction of~\cite{BCLS20} for motivations). 

\medskip 

It might be of interest to try to formulate a similar definition of weak involutive cubical $\omega$-category using M.Batanin and T.Leinster's operadic techniques, as already done for the globular case in~\cite{BeBe23}.  

\medskip 

A more ambitious future goal will be the exploration of equivalences between weak globular involutive \hbox{$\omega$-cat}\-e\-gories in~\cite{BeBe17} and the present weak cubical involutive $\omega$-categories, extending to the involutive weak category case famous results in~\cite{ABS02}. In this direction, one must first generalize to the strict $\omega$-category environment the (already quite involved) results obtained for strict involutive double categories and strict involutive globular 2-categories in~\cite{BCDM14}. 


\bigskip 

\emph{Notes and Acknowledgments:} 
P.Bertozzini thanks Starbucks Coffee (Langsuan, Jasmine City, Gaysorn Plaza, Emquartier Sky Garden) where he spent most of the time dedicated to this research project; he thanks Fiorentino Conte of ``The Melting Clock'' for the great hospitality during many crucial on-line dinner-time meetings.

\end{document}